\documentclass{article}

\usepackage{amsmath, mathrsfs, amssymb,amsthm,hyperref,tikz,tikz-3dplot,centernot}
\usepackage{tikz-cd}
\hypersetup{pdfstartview={XYZ null null 1.25}}
\usepackage[all]{xy}

\newcommand{\bw}{\bigwedge}

\newcommand{\bv}{\bigvee}

\newcommand{\pd}{p^\downarrow}

\newcommand{\Cp}{\Phi[\wp(P)]}

\newcommand{\LMD}{\textbf{MD}}
\newcommand{\cU}{\mathcal U}
\newcommand{\cV}{\mathcal V}

\theoremstyle{plain}
\newtheorem{thm}{Theorem}[section]
\newtheorem{prop}[thm]{Proposition}
\newtheorem{cor}[thm]{Corollary}
\newtheorem{lemma}[thm]{Lemma}
\theoremstyle{definition}
\newtheorem{defn}[thm]{Definition}
\newtheorem{example}[thm]{Example}
\numberwithin{equation}{section}

\begin{document}
\title{Closure operators, frames, and neatest representations}
\author{Rob Egrot}
\date{}

\maketitle

\begin{abstract}
Given a poset $P$ and a standard closure operator $\Gamma:\wp(P)\to\wp(P)$ we give a necessary and sufficient condition for the lattice of $\Gamma$-closed sets of $\wp(P)$ to be a frame in terms of the recursive construction of the $\Gamma$-closure of sets. We use this condition to show that given a set $\cU$ of distinguished joins from $P$, the lattice of $\cU$-ideals of $P$ fails to be a frame if and only if it fails to be $\sigma$-distributive, with $\sigma$ depending on the cardinalities of sets in $\cU$. From this we deduce that if a poset has the property that whenever $a\wedge(b\vee c)$ is defined for $a,b,c\in P$ it is necessarily equal to $(a\wedge b)\vee (a\wedge c)$, then it has an $(\omega,3)$-representation. This answers a question from the literature.     
\end{abstract}  

\section{Introduction}
Schein \cite{Sch72} defines a meet-semilattice $S$ to be distributive if it satisfies the first order definable condition that whenever $a\wedge(b\vee c)$ is defined, $(a\wedge b)\vee(a\wedge c)$ is also defined and the two are equal. A dual definition can, of course, be made for join-semilattices. Here for simplicity all semilattices are considered to be meet-semilattices, and the order duals to our results are left unstated.  Note that Schein's version of distributivity is strictly weaker than the notion of distributivity introduced by Gr\"atzer and Schmidt \cite{GraSchm62} (see for example \cite{Var75}), though they coincide for lattices. Schein's distributivity (hereafter referred to as $3$-distributivity, for reasons that will become clear later) is equivalent to the property of being embeddable into a powerset algebra via a semilattice homomorphism preserving all existing binary joins \cite{Bal69}. 

Schein's $3$-distributivity can be generalized to the concept of $\alpha$-distributivity for cardinals $\alpha$ (see definition \ref{D:dist}). This notion has been studied when $\alpha = n <\omega$ \cite{Hic78}, when $\alpha =\omega$ \cite{Var75,CorHic78}, and when $\alpha$ is any regular cardinal \cite{HicMon84}. Note that if $m,n\leq\omega$ with $m<n$ then $n$-distributivity trivially implies $m$-distributivity, but the converse is not true \cite{Kea97}. A recurring theme in these investigations is that a semilattice $S$ is $\alpha$-distributive if and only if the complete lattice of downsets of $S$ that are closed under existing $\alpha$-small joins is a frame (i.e. it satisfies the complete distributivity condition that $x\wedge \bv Y = \bv_Y(x\wedge y)$ for all elements $x$ and subsets $Y$). 

Notions relating to distributivity have also been studied in the more general setting of partially ordered sets (posets) \cite{Schm72a,HicMon84}. Note that the concepts in \cite{Schm72a} and \cite{HicMon84} are not equivalent, even for semilattices (that of \cite{Schm72a} is a straightforward generalization of the distributivity of Gr\"atzer and Schmidt \cite{GraSchm62}, while that of \cite{HicMon84} is more closely related to that of Schein \cite{Sch72}). 

As the transition from lattices to semilattices causes previously equivalent formulations of distributivity to diverge, so too does the transition from semilattices to posets. The following example is of particular interest to us. For a semilattice, being an $\omega$-distributive poset as defined in \cite{HicMon84} is equivalent to being an $\omega$-distributive semilattice in the sense used here \cite[proposition 2.3]{HicMon84}, which is in turn equivalent to being embeddable into a powerset algebra via a map preserving all existing finite meets and joins \cite{Bal69}. However, this is not true for arbitrary posets, as the $\omega$-distributivity of \cite{HicMon84} is strictly stronger than the powerset algebra embedding property \cite[example 4.2 and theorem 4.8]{HicMon84}.

The property of being embeddable into a powerset algebra via an embedding preserving meets and joins of certain cardinalities has been studied using the terminology \emph{representable} \cite{CheKem92,Kem93, Egr16} (see definition \ref{D:rep}). A notable departure from the semilattice case is that the first order theory of representable posets is considerably more complex. For example, while the class of $(m,n)$-representable posets is elementary for all $m,n$ with $2<m,n\leq\omega$ \cite{Egr16}, explicit first order axioms are not known, and the class cannot be finitely axiomatized \cite{Egrsub2}. This contrasts with the semilattice case where intuitive first order axioms are known, and only a finite number are required to ensure $(\omega,n)$-representability for finite $n$ \cite{Bal69}.   

In \cite{CheKem92} a condition for posets, which we will refer to as $\LMD$, generalizing the distributive property for semilattices is defined (see definition \ref{D:LMD}). The authors conjecture (\cite[section 3]{CheKem92}) that this condition is sufficient to ensure a poset has a representation preserving existing binary joins and all existing finite meets (they call this a \emph{neatest representation}). It is easily seen that this is not a \emph{necessary} condition for such a representation, e.g. \cite[example 1.5]{Egr16}). 

In the original phrasing of \cite[definition 1]{CheKem92} it is unclear whether a neatest representation  must preserve \emph{arbitrary} existing meets or only finite ones. From context we assume the latter, as the with the former definition the conjecture is false. To see this note that every Boolean algebra satisfies $\LMD$, and indeed the stronger condition that if $x\wedge \bv Y$ is defined then $\bv_Y(x\wedge y)$ is also defined and the two are equal (a Boolean algebra also satisfies the dual to this condition, so a complete Boolean algebra is structurally both a frame and a co-frame). However, it is well known that a Boolean algebra has a representation preserving arbitrary meets and/or joins if and only if it atomic \cite[corollary 1]{Abi71}. 

However, if we only demand that \emph{finite} meets are preserved then the conjecture is indeed correct, which we prove in this note. The main step in our proof is a result classifying the standard closure operators on $\wp(P)$ whose lattices of closed sets are frames as being precisely those that can be constructed using a certain recursive procedure (theorem \ref{T:main}). This can be viewed as a generalization of \cite[theorem 2.7]{HicMon84}. The proof of the conjecture about neatest representations is then an easy corollary.    

We note that the argument in the solution of this conjecture boils down to a straightforward extension of the forward direction of the argument for semilattices in \cite[theorem 2]{Hic78} to the poset setting. The result however is nevertheless perhaps surprising given that the class of posets with this kind of representation cannot be finitely axiomatized (see section \ref{S:obs} for a discussion of this).  

We proceed as follows. In section \ref{S:rad} we define introduce join-specifications and the concept of the radius of a standard closure operator, and assorted supporting terminology. In section \ref{S:lattices} we use this to prove the central theorem (theorem \ref{T:main}), and in the fourth section we prove a general version of the conjecture on neatest representations (corollary \ref{C:final}). The final section is devoted to a short discussion of the implications of this for the theory of poset representations, and some indirectly related questions in complexity theory.

\section{Join-specifications and standard closure operators}\label{S:rad}
We begin with a brief introduction to closure operators and their relationship with poset completions. This topic has been studied in detail, and a recent survey can be found in \cite{Ern09}. Nevertheless it will be useful to present some well known results with a consistent terminology, and to make explicit some facts that are only implicit elsewhere. We will give specific references when appropriate.

\begin{defn}[closure operator]
Given a set $X$ a closure operator on $X$ is a map from $\wp(X)$ to itself that is extensive, monotone and idempotent. I.e. such that for all $S,T\subseteq X$ we have:
\begin{enumerate}
\item $S\subseteq\Gamma(S)$,
\item $S\subseteq T\implies \Gamma(S)\subseteq \Gamma(T)$, and
\item $\Gamma(\Gamma(S))=\Gamma(S)$.
\end{enumerate} 
\end{defn} 

Given a set $S\subseteq P$ we denote $\{p\in P: p\leq s$ for some $s\in S\}$ by $S^\downarrow$, and $\pd$ is used a shorthand for $\{p\}^\downarrow$. Given a poset $P$ we are interested in closure operators on $P$. In particular we are interested in \emph{standard} closure operators, that is, closure operators such that $\Gamma(\{p\})=\pd$ for all $p\in P$. 

\begin{defn}[join-completion]
Given a poset $P$ a join-completion of $P$ is a complete lattice $L$ and an order embedding $e:P\to L$ such that $e[P]$ (the image of $P$ under $e$) is join-dense in $L$. 
\end{defn}

Given a standard closure operator $\Gamma:\wp(P)\to\wp(P)$, the $\Gamma$-closed sets form a complete lattice when ordered by inclusion (which we denote with $\Gamma[\wp(P)]$), and the map $\phi_\Gamma:P\to\Gamma[\wp(P)]$ defined by $\phi_\Gamma(p)=\pd$ is a join-completion of $P$ (we usually omit the subscript and just write $\phi$). If $I$ is an indexing set and $C_i$ is a $\Gamma$-closed set for all $i\in I$ then $\bw_I C_i= \bigcap_I C_i$ and $\bv_I C_i= \Gamma (\bigcup_I C_i)$ (see e.g. \cite[section 2.1]{Ern09}). Conversely, given a join-completion $e:P\to L$ the sets $e^{-1}[x^\downarrow]$ for $x\in L$ define the closed sets of a standard closure operator $\Gamma_e:\wp(P)\to\wp(P)$. This well known connection can be expressed as the following proposition. 

\begin{prop}
There is a dual isomorphism between the complete lattice of standard closure operators on $P$ (ordered by pointwise inclusion) and the complete lattice of ($e[P]$ preserving isomorphism classes of) join-completions of $P$ (ordered by defining $(e_1:P\to L_1)\leq (e_2:P\to L_2)$ if and only if there is an order embedding $\psi:L_1\to L_2$ such that $\psi\circ e_1 = e_2$. If $\psi$ exists it will necessarily preserve all existing meets).
\end{prop}
\begin{proof}
A direct argument is straightforward. See for example \cite[propositions 2.1 and 2.19]{Ern09}, or the introduction to \cite{Schm72a}, for equivalent results.
\end{proof} 

We are interested in join-completions $e:P\to L$ where the embedding $e$ preserves certain existing joins from $P$. Often this is done by making some uniform selection, such the joins of all sets smaller than some fixed cardinal (when they exist, see e.g. \cite{Schm72b}). We intend to be more general, and to that end we make the following definition. 

\begin{defn}[join-specification]\label{D:jspec} 
Let $P$ be a poset. Let $\cU$ be a subset of $\wp(P)$. Then $\cU$ is a join-specification (of $P$) if it satisfies the following conditions:
\begin{enumerate}
\item $\bv S$ exists in $P$ for all $S\in \cU$, 
\item $\{p\}\in \cU$ for all $p\in P$, and
\item $\emptyset\notin \cU$.
\end{enumerate}
\end{defn}

Definition \ref{D:jspec} is similar to that of a \emph{subset selection} (see e.g. \cite[section 2]{Ern09}). The difference is that we demand that the selection contains the singletons and that every selected set has a defined join. This serves to tidy up some of the later definitions. 

\begin{defn}[radius of $\cU$]\label{D:rad} 
Given a join-specification $\cU$ we define the radius of $\cU$ to be the smallest cardinal $\sigma$ such that $\sigma> |S|$ for all $S\in\cU$. 
\end{defn}

\begin{defn}[$\Gamma_\cU$]\label{D:G}
Given a join-specification $\cU$ with radius $\sigma$ and $S\subseteq P$ we define the following subsets of $P$ using transfinite recursion.
\begin{itemize}
\item $\Gamma_0(S) = S^\downarrow$.
\item If $\alpha+1$ is a successor ordinal then $\Gamma_{\alpha+1}(S)=\{\bv T : T\in\cU$ and $T\subseteq \Gamma_\alpha(S)\}^\downarrow$.
\item If $\lambda$ is a limit ordinal $\Gamma_\lambda(S)=\bigcup_{\beta<\lambda}\Gamma_\beta(S)$.
\end{itemize}
We define $\Gamma_\cU:\wp(P)\to\wp(P)$ by $\Gamma_\cU(S)=\Gamma_{\chi}(S)$ for all $S\in\wp(P)$, where $\chi$ is the smallest regular cardinal with $\sigma\leq\chi$.
\end{defn}

\begin{defn}[$\cU$-ideal] Given a join-specification $\cU$ we define a $\cU$-ideal of $P$ to be a down-set closed under joins from $\cU$. We define the empty set to be a $\cU$-ideal. 
\end{defn}
The following is a generalization of \cite[proposition 1.2]{HicMon84}.

\begin{prop}\label{P:trans}
If $\cU$ is a join-specification then $\Gamma_\cU$ is the standard closure operator taking $S\subseteq P$ to the smallest $\cU$-ideal containing $S$. 
\end{prop}
\begin{proof}
Let $S\in\wp(P)\setminus\{\emptyset\}$, and let $I$ be the smallest $\cU$-ideal containing $S$. Then we must have $\Gamma_\cU(S)\subseteq I$ by the closure requirements of $I$. It is easy to see that $\Gamma_\cU(S)$ is a down-set, so it remains only to show that $\Gamma_\cU(S)$ is closed under joins from $\cU$. So let $X\in\cU$ and suppose $X\subseteq \Gamma_\cU(S)$. Then by definition of $\Gamma_\cU$ we have $X\subseteq \bigcup_{\beta<\chi}\Gamma_\beta(S)$, so for each $x\in X$ there is some $\beta_x<\chi$ with $x\in\Gamma_{\beta_x}(S)$. Since $\chi$ is regular there must be $\beta'$ with $\beta_x\leq\beta'<\chi$ for all $x\in X$, and so $\bv X\in \Gamma_{\beta' +1}(S)\subseteq \Gamma_{\chi}(S)=\Gamma_\cU(S)$ as required. It is straightforward to show that the function taking sets to the smallest $\cU$-ideal containing them is a standard closure operator.
\end{proof}

\begin{prop}\label{P:morph}
Given a join-specification $\cU$, the canonical map $\phi:P\to\Gamma_\cU[\wp(P)]$ preserves arbitrary existing meets and the joins of all sets from $\cU$.
\end{prop} 
\begin{proof}
This is well known, but we give a short proof for the sake of completeness. First recall that arbitrary intersections of $\Gamma$-closed sets are also $\Gamma$-closed. So if $\bw S = t$ in $P$, then $\bw \phi[S] = \bigcap_S s^\downarrow = t^\downarrow = \phi(t)$ as required. Preservation of joins from $\cU$ follows from lemma \ref{L:equiv} and corollary \ref{C:relate}(2) below. 
\end{proof}

Note that $\phi$ may also preserve joins of sets that are not in $\cU$. Given a join-specification $\cU$ there will generally be more than one join-completion preserving the specified joins. We are interested in $\Gamma_\cU[\wp(P)]$, which is in fact the largest such join-completion (\cite{Schm74} attributes this result to \cite{Doc67}). This is easily seen by noting that if $e:P\to L$ is a join-completion such that $e(\bv S)=\bv e[S]$ for some $S\subseteq P$, then whenever $x\in L$ and $S\subseteq e^{-1}[x^\downarrow]$ we must have $\bv S\in e^{-1}[x^\downarrow]$ as otherwise $\bv e[S]\neq e(\bv S)$. So in particular if $e:P\to L$ preserves joins from $\cU$ then $e^{-1}[x^\downarrow]$ is a $\cU$-ideal for all $x\in L$. Since $L$ is isomorphic to $\{e^{-1}[x^\downarrow]:x\in L\}$ considered as a lattice ordered by inclusion the result follows as we can think of $L$ as being a subset of the set of all $\cU$-ideals.

\begin{defn}[$\cU_\Gamma$]\label{D:cUG}
Any standard closure operator $\Gamma:\wp(P)\to\wp(P)$ defines a join-specification $\cU_\Gamma$ by 
\begin{equation*}S\in\cU_\Gamma\iff \bv S \text{ exists and for all } \Gamma\text{-closed sets } C \text{ we have } S\subseteq C\implies \bv S\in C 
\end{equation*}  
\end{defn}

\begin{lemma}\label{L:equiv}
For $\cU_\Gamma$ as in definition \ref{D:cUG} and the canonical $\phi:P\to \Gamma[\wp(P)]$ we have 
\begin{equation*}S\in\cU_\Gamma\iff \bv S \text{ exists and } \phi(\bv S) = \bv \phi[S]
\end{equation*} 
\end{lemma}
\begin{proof}
Let $S\in \cU_\Gamma$. Then $\bv\phi[S]$ is the smallest $\Gamma$-closed set containing $S$, so must therefore contain $\bv S$, and is indeed equal to $(\bv S)^\downarrow=\phi(\bv S)$. Conversely, if $\bv S$ exists and $\bv \phi[S]=\phi(\bv S)$ then every $\Gamma$-closed set containing $S$ must contain $\bv S$, and so $S\in\cU_\Gamma$ by definition.
\end{proof}

The join-specifications of a poset $P$ are a subset of $\wp(\wp(P))$, and as this subset is closed under taking arbitrary unions and intersections they form a complete lattice when ordered by inclusion. This leads us to the following result. 
\begin{prop}\label{P:gal}
Let $\mathcal J$ be the lattice of join-specifications of $P$, and let $\mathcal C$ be the lattice of standard closure operators on $P$. Define $f:\mathcal J\to \mathcal C$ and $g:\mathcal C\to \mathcal J$ by  $f(\cU)=\Gamma_\cU$ and $g(\Gamma)=\cU_\Gamma$. Then $f$ and $g$ form a Galois connection between $\mathcal C$ and $\mathcal J$. I.e. for all $\Gamma\in \mathcal C$ and $\cU\in\mathcal J$ we have $\Gamma_\cU(T)\leq \Gamma(T)$ for all $T\in\wp(P)\iff \cU\subseteq \cU_\Gamma$.
\end{prop}
\begin{proof}
Let $\Gamma\in\mathcal C$ and let $\cU\in\mathcal J$. Suppose $\Gamma_\cU\leq \Gamma$. Since $\mathcal J$ is ordered by inclusion we let $S\in \cU$, and we aim to show that $S\in \cU_\Gamma$. Now, $\Gamma_\cU(S)$ is the smallest $\cU$-ideal containing $S$, so $\bv S\in \Gamma_\cU(S)$, and since $\Gamma_\cU(S)\leq \Gamma(S)$ this means any $\Gamma$-closed set containing $S$ must contain $\bv S$. But then by definition $S\in\cU_\Gamma$ so we are done. 

For the converse suppose $\cU\subseteq \cU_\Gamma$ and let $T\in\wp(P)$. We aim to show that $\Gamma_\cU(T)\subseteq \Gamma(T)$. Since $\Gamma_\cU(T)$ is the smallest $\cU$-ideal containing $T$ it is sufficient to show that $\Gamma(T)$ is also a $\cU$-ideal. But since $\cU\subseteq \cU_\Gamma$ this follows directly from the definition of $\cU_\Gamma$.
\end{proof}

\begin{cor}\label{C:relate}\mbox{}
\begin{enumerate}
\item[(1)] For all standard closure operators $\Gamma:\wp(P)\to\wp(P)$ we have $\Gamma_{\cU_\Gamma}\leq \Gamma$, but we do not necessarily have $\Gamma=\Gamma_{\cU_\Gamma}$.
\item[(2)] For all join-specifications $\cU$ of $P$ we have $\cU\subseteq \cU_{\Gamma_\cU}$, but we do not necessarily have $\cU= \cU_{\Gamma_\cU}$.
\item[(3)] If $\cU'= \cU_{\Gamma_\cU}$ then $\Gamma_\cU=\Gamma_{\cU'}$.
\item[(4)] If $\Gamma'= \Gamma_{\cU_\Gamma}$ then $\cU_\Gamma=\cU_{\Gamma'}$.
\end{enumerate}
\end{cor}
\begin{proof}
This all follows from the fact that $f$ and $g$ from proposition \ref{P:gal} form a Galois connection, with  examples \ref{E:neq} and \ref{E:neq2} witnessing lack of equality for parts (1) and (2) respectively.  
\end{proof}

\begin{example}\label{E:neq}
Let $P$ be the three element antichain $\{a,b,c\}$, and let the $\Gamma$-closed sets be $\emptyset, \{a\},\{b\},\{c\}$, and $\{a,b,c\}$. Then $\cU_\Gamma=\{\{a\},\{b\},\{c\}\}$ (so $\Gamma[\wp(P)]$ is the MacNeille completion of $P$), and the set of $\cU_\Gamma$-ideals of $P$ is just $\wp(P)$ (so $\Gamma_{\cU_\Gamma}[\wp(P)]$ is the Alexandroff completion composed of all down-sets of $P$ in this case). Then, for example, $\Gamma(\{a,b\})=\{a,b,c\}$, but $\Gamma_{\cU_\Gamma}(\{a,b\})=\{a,b\}$. 
\end{example}
Example \ref{E:neq} also demonstrates that not every standard closure operator arises from a join-specification. This is because the only join-specification on $P$ is $\{\{a\},\{b\},\{c\}\}$, and as we saw in the example the induced closure operator produces the Alexandroff completion, and not, for example, the MacNeille completion. This issue is also discussed in \cite[section 2]{Schm74}.

\begin{example}\label{E:neq2} 
Let $P$ be a poset containing elements $x,x',y,y',$ and $z$. Let the non-trivial orderings be $x<x'<z$, and $y<y'<z$, so $x\vee y = x'\vee y'=z$. Let $\cU=\{\{x\},\{x'\},\{y\},\{y'\},\{z\},\{x,y\}\}$ be a join-specification. Then $\{x',y'\}\in \cU_{\Gamma_\cU}\setminus \cU$.  
\end{example}

It follows from corollary \ref{C:relate} that different join-specifications can define the same closure operator. Given a standard closure operator $\Gamma$ arising from a join-specification, while there is not necessarily a minimal generating join-specification, definition \ref{D:rad} allows us to define a class of join-specifications generating $\Gamma$ whose sets have the smallest possible maximum size. This puts an upper bound on the number of iterations required in a recursive construction of $\Gamma$ (though we can't expect to do better than $\omega$, even if all the sets in the generating join-specification have bounded finite size). These minimal join-specifications (definition \ref{D:min} below) are also relevant when discussing the distributivity of the lattice of $\Gamma$-closed sets (see corollary \ref{C:bound}). 

\begin{defn}[radius of $\Gamma$]
Given a standard closure operator $\Gamma$ such that $\Gamma=\Gamma_\cU$ for some join-specification $\cU$, we define the radius of $\Gamma$ to be the minimum of $\{\chi:\chi$ is the radius of a join-specification $\cU'$ of $P$ with $\Gamma_{\cU'}=\Gamma\}$.
\end{defn}

\begin{defn}[minimal join-specification]\label{D:min}
A join-specification $\cU$ of $P$ is minimal if the radius of $\cU$ is equal to the radius of $\Gamma_\cU$. 
\end{defn}

The following technical lemma will be used in the next section.

\begin{lemma}\label{L:fix}
Let $\cU$ be a join-specification. Then the following hold for all $S\in\wp(P)$:
\begin{enumerate}
\item[(1)] If $\Gamma_\cU(S)=\pd$ then $p=\bv S$. 
\item[(2)] If $p=\bv S$ and $p\in\Gamma_\cU(S)$ then $S\in \cU_{\Gamma_\cU}$.
\end{enumerate}
\end{lemma}
\begin{proof}
For the first part note that $p$ must be an upper bound for $S$, so if $p\neq \bv S$ then $S$ has another upper bound $q$ with $p\not\leq q$. But $\Gamma_\cU(S)\subseteq \pd\cap q^\downarrow$ by proposition \ref{P:trans}, and thus $p\notin \Gamma_\cU(S)$, which would be a contradiction. 
For the second part note that since $\bv S=p$, by definition \ref{D:cUG} we have $S\in \cU_{\Gamma_\cU}\iff$ for all $\Gamma_\cU$-closed sets $C$ we have $S\subseteq C\implies p\in C$. Since $\Gamma_\cU(S)$ is the smallest $\Gamma_\cU$-ideal containing $S$,  if $p\in \Gamma_\cU(S)$ then every $\Gamma_\cU$-ideal containing $S$ must also contain $p$ and we are done.  
\end{proof}

\section{When is a lattice of $\cU$-ideals a frame?}\label{S:lattices}
As mentioned in the introduction, definitions of distributivity in semilattices modeled on that of Schein \cite{Sch72} give rise to results that can be stated in our terminology as \emph{a semilattice $S$ is $\alpha$-distributive if and only if the lattice of $\cU_\alpha$-ideals of $S$ is $\alpha$-distributive}, where $\cU_\alpha$ contains all sets smaller than $\alpha$ whose joins are defined (where $\alpha$ is some finite or regular cardinal \cite{Hic78,CorHic78,HicMon84}). Moreover, it turns out that if this lattice of $\cU_\alpha$-ideals is $\alpha$-distributive then it will be a frame (see definition \ref{D:dist}  below). We can extend this to posets and arbitrary join-specifications, in the sense that if the lattice of $\cU$-ideals of $P$ fails to be a frame then it must be because distributivity fails for the embedded images of some element of $P$ and some set in $\cU$ (see corollary \ref{C:bound''}). The key result is theorem \ref{T:main}, which can be seen as a partial generalization of \cite[theorem 2.7]{HicMon84} to arbitrary join-specifications. 

\begin{defn}[$\alpha$-distributive]\label{D:dist} 
Given a cardinal $\alpha$ we say a lattice (or a semilattice) $L$ is $\alpha$-distributive if given $\{x\}\cup Y\subseteq L$ such that $|Y|<\alpha$, if $x\wedge \bv Y$ exists then $\bv_Y(x\wedge y)$ also exists and the two are equal. Note that when $3\leq n\leq\omega$, in the lattice case $n$-distributivity is just distributivity. When $L$ is a complete lattice and $L$ is $\alpha$-distributive for all $\alpha$ we say $L$ is a \emph{frame}.
\end{defn}

From now on we fix a minimal join-specification $\cV$ and we define $\Phi=\Gamma_\cV$. Let $\sigma$ be the radius of $\cV$ and let $\chi$ be the smallest regular cardinal with $\sigma\leq\chi$. Similarly let $\sigma'$ be the radius of $\cU_\Phi$ (recall definition \ref{D:cUG}), and let $\chi'$ be the smallest regular cardinal with $\sigma'\leq\chi'$.

\begin{defn}[$\Upsilon$]\label{D:U}
Given $S\subseteq P$ we define the following subsets of $P$ using transfinite recursion.
\begin{itemize}
\item $\Upsilon_0(S) = S^\downarrow$.
\item If $\alpha+1$ is a successor ordinal then $\Upsilon_{\alpha+1}(S)=\{\bv T : T\in\cU_{\Phi}$ and $T\subseteq \Upsilon_\alpha(S)\}$.
\item If $\lambda$ is a limit ordinal $\Upsilon_\lambda(S)=\bigcup_{\beta<\lambda}\Upsilon_\beta(S)$.
\end{itemize}
We define $\Upsilon:\wp(P)\to\wp(P)$ by $\Upsilon(S)=\Upsilon_{\chi'}(S)$ for all $S\in\wp(P)$.
\end{defn}
This definition differs from definition \ref{D:G} in that we do not close downwards during the successor steps. Also note the use of $\chi'$ in place of $\chi$. This is important in the proof of corollary \ref{C:bound}. Note that $\Upsilon$ will not necessarily be a closure operator as it may not be idempotent.

\begin{lemma}\label{L:sub}
$\Upsilon(S)\subseteq \Phi(S)$ for all $S\in \wp(P)$.
\end{lemma}
\begin{proof}
We have $\Phi=\Gamma_\cV=\Gamma_{\cU_\Phi}$ by proposition \ref{C:relate}(3), and by definition of $\Upsilon$ we have $\Upsilon(S)\subseteq \Gamma_{\cU_\Phi}(S)$ for all $S\in\wp(P)$.
\end{proof}

\begin{lemma}\label{L:dist}
Let $J$ be an indexing set, and let  $I,K_j\in\Cp$ for all $j\in J$. Then $I\cap \Upsilon(\bigcup_J K_j) \subseteq \Phi(\bigcup_J(I\cap K_j))$.
\end{lemma}
\begin{proof}
We note that $\Upsilon(\bigcup_J K_j)=\bigcup_{\alpha<\chi'}\Upsilon_\alpha(\bigcup_J K_j)$. We proceed by showing that $I\cap\Upsilon_\alpha(\bigcup_J K_j)\subseteq \Phi(\bigcup_J(I\cap K_j))$ for all $\alpha$ using transfinite induction on $\alpha$. If $\alpha=0$ the result is trivial so consider the successor ordinal $\alpha+1$ and assume the appropriate inductive hypothesis. Let $p\in I\cap \Upsilon_{\alpha+1}(\bigcup_J K_j)$. Then $p=\bv T$ for some $T\in\cU_\Phi$ with $T\subseteq \Upsilon_{\alpha}(\bigcup_J K_j)$, and so since $p\in I$ we have $T\subseteq I\cap \Upsilon_{\alpha}(\bigcup_J K_j)$. By the inductive hypothesis this means $T\subseteq \Phi(\bigcup_J(I\cap K_j))$, and thus by definition of $\cU_\Phi$ we have $p\in \Phi(\bigcup_J(I\cap K_j))$ as required. The limit case is trivial.
\end{proof}

\begin{thm}\label{T:main}
The following are equivalent:
\begin{enumerate} 
\item $\Phi(S)=\Upsilon(S)$ for all $S\in\wp(S)$.
\item $\Cp$ is a frame when considered to be a lattice ordered by inclusion.
\end{enumerate}
\end{thm}
\begin{proof}
($1\implies 2$). Let $J$ be an indexing set and let $I,K_j\in \Cp$ for all $j\in J$. We must show that $I\cap\Phi(\bigcup_J K_j)=\Phi(\bigcup_J(I\cap K_j))$. Since we are assuming $\Phi=\Upsilon$ and the right side is always included in the left it remains only to show that $I\cap\Upsilon(\bigcup_J K_j)\subseteq\Phi(\bigcup_J(I\cap K_j))$, and this is lemma \ref{L:dist}.

($2\implies 1$). We note that $\Phi_0(S)=\Upsilon_0(S)$ for all $S\in\wp(P)$ by definition, and we proceed by using transfinite induction to show that $\Phi_\alpha(S)\subseteq \Upsilon_\alpha(S)$ for all cardinals $\alpha$ and for all $S\in\wp(P)$. As $\Upsilon(S)\subseteq \Phi(S)$ by lemma \ref{L:sub} the result then follows. 

Let $S\in \wp(P)$ and suppose $\Phi_\alpha(S)\subseteq\Upsilon_\alpha(S)$ for some cardinal $\alpha$. Let $T\subseteq \Phi_\alpha(S)$ and suppose $T\in\cU$. Let $p\leq \bv T$. Then in $\Phi[\wp(P)]$ we have $\pd\cap \bv _T t^\downarrow=\pd\cap (\bv T)^\downarrow = \pd$ by proposition \ref{P:morph}. Since $\Phi[\wp(P)]$ is a frame we also have $\bv_T (\pd \cap t^\downarrow)= \pd$, and thus $\bv_T (\pd \cap t^\downarrow)=\Phi(\bigcup_T(\pd \cap t^\downarrow))=\Phi(T^\downarrow\cap \pd)=\pd$. So by lemma \ref{L:fix}(1) we have $p=\bv (T^\downarrow\cap \pd)$, and thus by lemma \ref{L:fix}(2) we have $T^\downarrow\cap \pd\in\cU_{\Phi}$. But $T^\downarrow\subseteq\Phi_\alpha(S)\subseteq \Upsilon_\alpha(S)$, and thus $T^\downarrow\cap \pd\subseteq \Upsilon_\alpha(S)$, and so $p\in \Upsilon_{\alpha+1}(S)$ as required. The limit case is trivial and so we are done.     
\end{proof}

\begin{cor}\label{C:bound}
If $\Cp$ is not a frame then there is $T\in\cV$ and $p\in P$ such that
\begin{enumerate}
\item $p\leq \bv T$, and 
\item $\phi(p)\wedge\bv \phi[T]\neq\bv_T(\phi(p)\wedge \phi(t))$ in $\Cp$. 
\end{enumerate}
Here $\phi$ is the canonical map from $P$ into $\Phi[\wp(P)]$ taking $p$ to $\pd$.
\end{cor}
\begin{proof}
By theorem \ref{T:main} if $\Cp$ is not a frame then there is $S\in \wp(P)$ with $\Phi(S)\neq\Upsilon(S)$. Let $\alpha$ be the smallest cardinal such that $\Phi_\alpha(S)\not\subseteq \Upsilon(S)$. This exists by lemma \ref{L:sub} and the assumption that $\Phi(S)\neq\Upsilon(S)$. Moreover, $\alpha$ cannot be $0$ or a limit cardinal so $\alpha=\beta + 1$ for some $\beta$. Choose any $p\in\Phi_{\alpha}(S)\setminus \Upsilon(S)$. Then by minimality of $\alpha$ we must have $p\notin\Phi_\beta(S)$, and $p\leq\bv T$ for some $T\in\cV$ with $T\subseteq \Phi_\beta(S)$. In $\Phi[\wp(P)]$ we have $\pd\cap \bv_T t^\downarrow = \pd$. However, if $\pd=\bv_T(\pd\cap t^\downarrow)=\Phi(T^\downarrow\cap\pd)$, then $p=\bv(T^\downarrow\cap \pd)$ and $T^\downarrow\cap \pd\in\cU_\Phi$ by lemma \ref{L:fix}. Since $T^\downarrow\cap \pd\subseteq \Phi_\beta(S)$, by choice of $\alpha$ we have $T^\downarrow\cap \pd\subseteq \Upsilon(S)$. Since $\Upsilon(S)$ is closed under joins from $\cU_\Phi$ (by an argument similar to that in the proof of proposition \ref{P:trans}), and $p\notin\Upsilon(S)$ by choice of $p$, we must therefore have $\pd\neq\bv_T(\pd\cap t^\downarrow)$ to avoid contradiction.   
\end{proof}

Here the minimality of $\cV$ is relevant as it gives a smaller upper bound on the possible size of a cardinal $\alpha$ for which $\alpha$-distributivity can fail. Two immediate consequences of corollary \ref{C:bound} are the following.

\begin{cor}\label{C:bound''}
If $\Cp$ is not a frame then $\Cp$ must fail to be $\sigma$-distributive (recall that $\sigma$ is the radius of $\cV$).
\end{cor}
\begin{proof}
If $\Cp$ is not a frame then distributivity fails for $\phi(p)\wedge\bv \phi[T]$ for some $T\in\cV$, and $|T|<\sigma$ by definition. 
\end{proof}

\begin{cor}\label{C:bound'}
Given a poset $P$ the lattice of all down-sets of $P$ closed under existing finite joins is a frame if and only if it is distributive.
\end{cor}
\begin{proof}
This lattice is produced by the closure operator arising from the join-specification containing all non-empty finite sets with defined joins. The radius of this join-specification is $\omega$, and the result then follows from corollary \ref{C:bound''}.
\end{proof}

In the case where $\kappa$ is a regular cardinal and $\cV=\cU_\kappa=\{S\in\wp(P)\setminus\{\emptyset\}:|S|<\kappa$ and $\bv S$ exists$\}$ theorem \ref{T:main} and its corollaries are superseded by \cite[theorem 2.7]{HicMon84}. Indeed, in this case the mentioned theorem shows, using our notation, that $\Cp$ is a frame if and only if $\Phi(S)=\{\bv T: T\in \cV$ and $T\subseteq S^\downarrow\}$ for all $S\in \wp(P)$. 

Various distributivity properties for join-completions of posets (and quasiorders) are investigated in \cite{ErnWil83}. See in particular \cite[theorem 2.1]{ErnWil83} for a summary of other properties equivalent to the join-completion corresponding to a given standard closure operator being a frame.

\section{Sentences guaranteeing representability}\label{S:sent}

\begin{defn}[$(\alpha,\beta)$-representable]\label{D:rep}
For cardinals $\alpha$ and $\beta$ a poset $P$ is $(\alpha,\beta)$-representable if there is an embedding $h:P\to F$, where $F$ is a powerset algebra, such that $h$ preserves meets of sets with cardinalities strictly less than $\alpha$, and joins of sets with cardinalities strictly less than $\beta$. When $\alpha=\beta$ we just write $\alpha$-representable.
\end{defn}

The following definition appears under a slightly different name as \cite[definition 2]{CheKem92}.

\begin{defn}[$\LMD$]\label{D:LMD}
A poset $P$ is $\LMD$ (meet distributive) if for all $a,b,c\in P$, if $a\wedge (b\vee c)$ is defined then $(a\wedge b)\vee (a\wedge c)$ is also defined and the two are equal.  
\end{defn}

Given a cardinal $\alpha\geq 3$ we can generalize this to the following.

\begin{defn}[$\LMD_\alpha$]
$P$ is $\LMD_\alpha$ if whenever $\{a\}\cup X\subseteq P$, $|X|<\alpha$ and $a\wedge\bv X$ is defined in $P$, we have $\bv_X(a\wedge x)$ is also defined in $P$ and the two are equal. 
\end{defn}

Note that for semilattices $\LMD_\alpha$ is equivalent to $\alpha$-distributivity. We use our terminology to avoid confusion with the distributivity for posets defined in \cite{HicMon84}. Given a cardinal $\gamma$ suppose $\cV=\cU_\gamma$ is the set of all sets $S\subseteq P$ such that $|S|< \gamma$ and $\bv S$ exists in $P$, and define $\Phi$ and $\Upsilon$ for this $\cV$ as in section \ref{S:lattices}. Note that in section \ref{S:lattices} we say $\cV$ is minimal, but minimality is not essential in the definitions of $\Phi$ and $\Upsilon$, or the theory developed therein. So the fact that $\cU_\gamma$ may not be minimal is not a problem here. We have the following theorem.

\begin{thm}\label{T:frame}
Let $\gamma$ be any cardinal strictly greater than 2. If $P$ is $\LMD_\gamma$ then $\Cp$ is a frame.
\end{thm}
\begin{proof}
By corollary \ref{C:bound} if $\Cp$ fails to be a frame there must be $T\in \cV$ and $p\in P$ such that $p\leq \bv T$ and $\phi(p)\wedge\bv \phi[T]\neq\bv_T(\phi(p)\wedge \phi(t))$ in $\Cp$. But by definition of $\cV$ we must have $|T|<\gamma$, and so by proposition \ref{P:morph} and the assumption that $P$ is $\LMD_\gamma$ we have $\phi(p)\wedge\bv \phi[T] = \phi(p\wedge \bv T)=\phi(\bv_T(p\wedge t))=\bv_T(\phi(p)\wedge\phi(t))$, so $\Cp$ must be a frame after all.   
\end{proof}

\begin{cor}\label{C:rep}
Every $\LMD$ poset has an $(\omega,3)$-representation.
\end{cor}
\begin{proof}
Let $\cV=\cU_3$ and let $\Phi$ be defined as in section \ref{S:lattices}. Then as $\Cp$ is distributive it embeds into a powerset algebra $F$ via a map preserving finite meets and joins. Since $\phi:P\to \Cp$ preserves binary joins and arbitrary meets we obtain an $(\omega,3)$-representation for $P$ by composing $\phi$ with the embedding of $\Cp$ into $F$.   
\end{proof}

By setting the value of $\gamma$ appropriately we also obtain the following result.

\begin{cor}\label{C:final}
If $3\leq n\leq \omega$ then every $\LMD_n$ poset has an $(\omega,n)$-representation.
\end{cor}

When $\alpha > \omega$ we do \emph{not} obtain a result corresponding to corollary \ref{C:final} because a frame is not necessarily $(\omega,\alpha)$-representable in this case. We know this because for every $\alpha>\omega$, every non-atomic countable Boolean satisfies $\LMD_\alpha$ but is not $(\omega,\alpha)$-representable, as discussed in the introduction.

Note that it follows from theorem \ref{T:frame} and \cite[theorem 2.7]{HicMon84} that if $\kappa$ is a regular cardinal we have $\LMD_\kappa\implies \kappa$-distributive in the sense of \cite{HicMon84} (which for convenience we shall dub $\mathbf{HM}_\kappa$). Also by \cite[theorem 2.7]{HicMon84}, being $\mathbf{HM}_\kappa$ is equivalent to having $\Phi(S)=\{\bv T: T\in \cU_\kappa$ and $T\subseteq S^\downarrow\}$ for all $S\in \wp(P)$ (as mentioned in the passage following corollary \ref{C:bound'}). Using this we can show that $\mathbf{HM}_\kappa\centernot\implies\LMD_\kappa$ for all regular $\kappa$, as example \ref{E:nimply} below provides a poset that is $\mathbf{HM}_\kappa$, but fails to be $\LMD_\kappa$, for all regular $\kappa$.

\begin{example}\label{E:nimply}
Let $P$ be the poset in the diagram below and let $\kappa$ be any regular cardinal. Given $S\in\wp(P)$ define $\Phi'(S)=\{\bv T: T\in \cU_\kappa$ and $T\subseteq S^\downarrow\}$. Then $a\wedge(b\vee c)=c$, but $a\wedge b$ does not exist. So $P$ fails to be $\LMD_3$, and thus fails to be $\LMD_\kappa$. However, the only non-trivial non-principal downsets of $P$ are $\{b,c\}$ and $\{a,b,c\}$, and $\Phi(\{b,c\})=(b\vee c)^\downarrow=\Phi'(\{b,c\})$, and $\Phi(\{a,b,c\})=P=\Phi'(\{a,b,c\})$, and so $P$ is $\mathbf{HM}_\kappa$ for all regular $\kappa$.   
\[\xymatrix{& \bullet & \bullet_a\\
\bullet_b\ar@{-}[ur] & \bullet_c\ar@{-}[ur]\ar@{-}[u] 
}\] 
\end{example}

Combining the preceding discussion with \cite[example 4.2 and theorem 4.8]{HicMon84} we obtain the following chain of strictly one way implications.

\begin{equation*}
\LMD_\omega\implies \mathbf{HM}_\omega \implies \omega\text{-representable}
\end{equation*}

\section{Obstructions to representability}\label{S:obs}
Another way to view these results is to think about obstructions to a poset being, for example, $3$-representable. For the sake of this discussion we say $P$ has a \emph{triple} $(a,b,c)$ if there are $a,b,c\in P$ such that $a\wedge(b\vee c)$ is defined. If $(a,b,c)$ is a triple of $P$ and $(a\wedge b)\vee (a\wedge c)$ is defined but not equal to $a\wedge(b\vee c)$ we call $(a,b,c)$ a \emph{split triple}, and we say a triple $(a,b,c)$ is an \emph{indeterminate triple} if $(a\wedge b)\vee (a\wedge c)$ is not defined. 

It's easy to see that having a split triple is a sufficient condition for a poset $P$ to fail to be $3$-representable. Moreover, existence of a split triple can be defined in first order logic, so if having a split triple were also a necessary condition for a poset to fail to be $3$-representable it would follow that the class of $3$-representable posets is finitely axiomatizable. This is not the case \cite{Egrsub2}, so there must be other, less obvious obstructions to $3$-representability. This is not surprising, as even in the simpler semilattice case the existence of a split triple is not necessary for failure of $3$-representability. However, the semilattices that fail to be $(\omega,3)$-representable can be characterized as those that contain either a split triple or an indeterminate triple (this is \cite[theorem 2.2]{Bal69} phrased in the terminology of triples), which is also a first order property. Since the class of $(m,n)$-representable posets is elementary for all $m$ and $n$ with $2<m,n\leq \omega$, it follows that the class of posets that fail to be $3$-representable cannot be axiomatized in first order logic at all (otherwise it would be finitely axiomatizable, in contradiction with \cite{Egrsub2}). This contrasts starkly with the intuitive finite axiomatization of the semilattice case.  

Putting this another way, let $\mathcal{S}$ be the class of posets containing a split triple, let $\mathcal{I}$ be the class of posets containing an indeterminate triple but no split triple, let $\mathcal{L}$ be the class of $\LMD$-posets, and let $\mathcal{R}$ be the class of $3$-representable posets. Then $\mathcal{S}$, $\mathcal{I}$, and $\mathcal{L}$ are all basic elementary and partition the class of all posets. Moreover, $\mathcal{L}\subset \mathcal{R}$, and $\mathcal{S}\subset \bar{\mathcal{R}}$. However, $\mathcal{I}\cap\mathcal{R}$ is elementary but not finitely axiomatizable, and $\mathcal{I}\cap\bar{\mathcal{R}}$ is not even elementary. This contrasts with the semilattice case where $\mathcal{L}$ and $\mathcal{R}$ coincide. 

On the other hand, one reason we might expect these obstructions to representability to defy simple characterization comes from computational complexity theory. It was shown in \cite{VanAta} that the problem of deciding whether a finite poset has an $(m,n)$-representation  is $\mathbf{NP}$-complete for countable $m,n$ with at least one greater than $3$. So if there were a simple characterization of the finite posets that fail to be $(m,n)$-representable, such as exists in the semilattice case, we could potentially use this to prove that this decision problem is also in $\mathbf{coNP}$, which would imply the unexpected coincidence $\mathbf{NP}=\mathbf{coNP}$. More precisely, by Fagin's theorem we have $\mathbf{NP}=\mathbf{coNP}$ if and only if the finite posets that fail to be $(m,n)$-representable (for suitable $m$ and $n$) have an existential second order characterization.

\bibliographystyle{srtnumbered}

\begin{thebibliography}{10}
\newcommand{\enquote}[1]{`#1'}

\bibitem{Abi71}
A.~Abian.
\newblock \enquote{Boolean rings with isomorphisms preserving suprema and
  infima}.
\newblock {\em J. Lond. Math. Soc. (2)\/} {\bf 3} (1971), 618--620.

\bibitem{Bal69}
R.~Balbes.
\newblock \enquote{A representation theory for prime and implicative
  semilattices}.
\newblock {\em Trans. Amer. Math. Soc.\/} {\bf 136} (1969), 261--267.

\bibitem{CheKem92}
Y.~Cheng and P.~Kemp.
\newblock \enquote{Representation of posets}.
\newblock {\em Zeitschr. f. math. Logik und Grundlagen d. Math\/} {\bf 38}
  (1992), 269--276.

\bibitem{CorHic78}
W.~Cornish and R.~Hickman.
\newblock \enquote{Weakly distributive semilattices}.
\newblock {\em Acta. Math. Acad. Sci. Hungar.\/} {\bf 32} (1978), 5--16.

\bibitem{Doc67}
H.~P. Doctor.
\newblock {\em Extensions of a partially ordered set\/} (1967).
\newblock Thesis (Ph.D.)--McMaster University (Canada).

\bibitem{Egrsub2}
R.~Egrot.
\newblock \enquote{No finite axiomatizations for posets embeddable into
  distributive lattices}.
\newblock {\em preprint\/} ArXiv:1610.00858.

\bibitem{Egr16}
R.~Egrot.
\newblock \enquote{Representable posets}.
\newblock {\em J. Appl. Log.\/} {\bf 16} (2016), 60--71.

\bibitem{Ern09}
M. Ern{\'e}.
\newblock \enquote{Closure}.
\newblock In {\em Beyond topology\/}, {\em Contemp. Math.\/}, Volume 486 (Amer.
  Math. Soc., Providence, RI, 2009), 163--238.

\bibitem{ErnWil83}
M. Ern{\'e} and G. Wilke.
\newblock \enquote{Standard completions for quasiordered sets}.
\newblock {\em Semigroup Forum\/} {\bf 27}~(1-4) (1983), 351--376.

\bibitem{GraSchm62}
G.~Gr\"atzer and E.T. Schmidt.
\newblock \enquote{On congruence lattices of lattices}.
\newblock {\em Acta Math. Acad. Sci. Hungar.\/} {\bf 13} (1962), 179--185.

\bibitem{Hic78}
R.~Hickman.
\newblock \enquote{Distributivity in semilattices}.
\newblock {\em Acta Math. Acad. Sci. Hungar.\/} {\bf 32} (1978), 35--45.

\bibitem{HicMon84}
R.C. Hickman and G.P. Monro.
\newblock \enquote{Distributive partially ordered sets}.
\newblock {\em Fund. Math.\/} {\bf 120} (1984), 151--166.

\bibitem{Kea97}
K.~Kearnes.
\newblock \enquote{The class of prime semilattices is not finitely
  axiomatizable}.
\newblock {\em Semigroup Forum\/} {\bf 55} (1997), 133--134.

\bibitem{Kem93}
P.~Kemp.
\newblock \enquote{Representation of partially ordered sets}.
\newblock {\em Algebra Universalis\/} {\bf 30} (1993), 348--351.

\bibitem{Sch72}
B.~Schein.
\newblock \enquote{On the definition of distributive semilattices}.
\newblock {\em Algebra Universalis\/} {\bf 2} (1972), 1--2.

\bibitem{Schm74}
J.~Schmidt.
\newblock \enquote{Each join-completion of a partially ordered set is the
  solution of a universal problem}.
\newblock {\em J. Austral. Math. Soc.\/} {\bf 17} (1974), 406--413.

\bibitem{Schm72b}
J. Schmidt.
\newblock \enquote{Universal and internal properties of some completions of
  {$k$}-join-semilattices and {$k$}-join-distributive partially ordered sets}.
\newblock {\em J. Reine Angew. Math.\/} {\bf 255} (1972), 8--22.

\bibitem{Schm72a}
J. Schmidt.
\newblock \enquote{Universal and internal properties of some extensions of
  partially ordered sets}.
\newblock {\em J. Reine Angew. Math.\/} {\bf 253} (1972), 28--42.

\bibitem{VanAta}
C. Van Alten. 
\newblock \enquote{Embedding Ordered Sets into Distributive Lattices}.
\newblock {\em Order\/} {\bf 33} (2016), 419--427.
 

\bibitem{Var75}
J.C. Varlet.
\newblock \enquote{On separation properties in semilattices}.
\newblock {\em Semigroup Forum\/} {\bf 10} (1975), 220--228.

\end{thebibliography}

\end{document}